\documentclass[1p]{elsarticle}

\usepackage{lineno}
\modulolinenumbers[5]

\journal{ }

\usepackage{mathrsfs}
\usepackage{amsfonts}
\usepackage[notref,notcite]{showkeys}
\usepackage{enumerate}
\usepackage{latexsym}
\usepackage[all,cmtip]{xy}
\usepackage{verbatim}

\usepackage{ifpdf}
\ifpdf
 \usepackage[colorlinks,final,backref=page,hyperindex]{hyperref}
\else
  \usepackage[colorlinks,final,backref=page,hyperindex,hypertex]{hyperref}
\fi

\usepackage{amsmath,amssymb,xspace,amsthm}

\numberwithin{equation}{section}

\def\gl{\mathfrak{gl}}

\def\fa{\mathfrak{a}}
\def\fg{\mathfrak{g}}

\def\fri{\mathfrak{i}}

\def\fq{\mathfrak{q}}

\def\fJ{\mathfrak{J}}
\def\fL{\mathfrak{L}}

\def\cF{\mathcal{F}}

\def\cU{\mathcal{U}}
\def\cI{\mathcal{I}}
\def\cT{\mathcal{T}}

\def\sJ{\mathscr{J}}

\def\bC{\mathbb{C}}
\def\bN{\mathbb{N}}
\def\bZ{\mathbb{Z}}

\def\Der{\mathrm{Der}}

\def\supp{\mathrm{supp}}
\def\ann{\mathrm{ann}}
\def\Ker{\mathrm{Ker}}

\def\tW{\widetilde{W}}

\newcommand\ol{\overline}
\newcommand\wt{\widetilde}

\newtheorem{theo}{{Theorem}}[section]
\newtheorem{lemm}[theo]{Lemma}

\newtheorem{defi}[theo]{Definition}

\newtheorem{prop}[theo]{Proposition}
\newtheorem{exam}[theo]{Example}

\allowdisplaybreaks

\begin{document}

\begin{frontmatter}



\title{Weight modules for map (super)algebra related to the Virasoro algebra}


\author{Yan-an Cai, Rencai L\"{u}, Yan Wang}

\begin{abstract}
We classify Jet modules for the Lie (super)algebras $\fL=W\ltimes(\fg\otimes\bC[t,t^{-1}])$, where $W$ is the Witt algebra and $\fg$ is a Lie superalgebra with an even diagonlizable derivation. Then we give a concept method to classify all simple cuspidal modules for $\fL$ and the map superalgebras, which are of the form $\fL\otimes R$, where $R$ is a Noetherian unital supercommutative associative superalgebra.
\end{abstract}


\end{frontmatter}


\section{Introduction}

We denote by $\bZ, \bZ_+, \bN, \bC$ and $\bC^*$ the sets of all integers, non-negative integers, positive integers, complex numbers, and nonzero complex numbers, respectively. All vector spaces and algebras in this paper are over $\bC$. We denote by $U(L)$ the universal enveloping algebra of the Lie (super)algebra $L$ over $\bC$. Also, we denote by $\delta_{i,j}$ the Kronecker delta. Throughout this paper, by subalgebras, submodules for Lie superalgebras we mean subsuperalgebras and subsupermodules respectively.

Let $A=\bC[t,t^{-1}]$. The Witt algebra $W=\Der\bC[t, t^{-1}]$ has a basis $\{d_n=t^{n+1}\frac{d}{dt}\,|\, n\in \mathbb{Z}\}$ with Lie brackets given by
$$[d_m, d_n]=(n-m)d_{n+m}.$$
It is well known that the Virasoro algebra, the universal central extension of the Witt algebra $W$, have been widely used in many physics areas and mathematical branches, and the Virasoro algebra plays a key role in representation theory of the affine Lie algebras. An important class of modules for the Virasoro algebra are the so-called quasifinite modules (or Harish-Chandra modules), which were classified by Mathieu in \cite{Ma}.

Many generalizations of the Virasoro algebra an other closely related algebras have been considered by several authors. These include, but are not limited to, the higher rank Virasoro algebras \cite{LZ,Maz,S1,S2}, the generalized Virasoro algebras \cite{BZ,GLZ1,HWZ}, the twisted Heisenberg-Virasoro algebra \cite{L, LZ1}, $W(2,2)$ \cite{LGZ}, the Schr\"{o}dinger-Virasoro algebra \cite{L}, the loop-Virasoro algebra \cite{GLZ2}, the affine-Virasoro algebra \cite{GHL},  the map Virasoro algebras \cite{Rao, NS, NSS, S}, Lie algebra of vector fields on a torus $W_n$ \cite{BF} and references therein.

Among qusifintie modules, there is an important class of modules, called the uniformly bounded or cuspidal modules, whose weight spaces have uniformly bounded dimension. Billig introduced Jet modules when studying cuspidal $W_n$-modules in \cite{B}.

Let $\fg$ be a Lie superalgebra and $d$ be an even diagonalizable derivation on $\fg$, that is, $\fg$ has a basis $\{x_s\,|\, s\in S\}$ such that $d(x_s)=\beta_sx_s$. Fix this basis for $\fg$. Let $\widehat{\fg}$ be the Lie superalgebra $\bC d+\fg$ with bracket $[d,x]=d(x), \forall x\in\fg$. Define a Lie superalgebra $\fL$ associated to $\fg$ as follows: $\fL=W\ltimes(\fg\otimes A)$ with brackets
\begin{align*}
&[d_i,x_s\otimes t^k]=(k+i\beta_s)x_s\otimes t^{i+k},\\
&[x\otimes t^k, y\otimes t^j]=[x,y]\otimes t^{k+j},
\end{align*}
where $i,j,k\in\bZ, x,y\in\fg$. Write $x(k)$ for $x\otimes t^k$. Let $I=\fg\otimes A$.

One goal of this paper is to classify Jet modules for $\fL$. These algebras include, but are not limited to, the  $W$-algebras $W(a,b)$, the differentially simple Lie superalgebras, the affine-Virasoro algebras. Another goal of this paper is to classify all simple cuspidal modules for $\fL$ and  the map superalgebras, which are of the form $\fL\otimes R$, where $R$ is a Noetherian unital supercommutative associative superalgebra.

The paper is organized as follows. In section \ref{pre}, we collect come basic notations and results for our study. Jet modules for $\fL$ are classified in section \ref{Jet}. Simple cuspidal modules for $\fL$ and the map super algebras $\fL\otimes R$ are classified in section \ref{cuspL} and section \ref{cuspmap} respectively. Finally, in section \ref{example} we present the results for some interesting algebras.

\section{Preliminaries}
\label{pre}

In this section, we collect some basic results for our study.

Denote by $\tW=W\ltimes A$ the extended Witt algebra. A $\tW$-module $V$ is called an $AW$-module if $A$ acts associatively, i.e., $t^0 v=v, fgv=f(gv),\forall f,g\in A, v\in V$.
Let $L$ be any Lie superalgebra containing $W$. An $L$-module $V$ is called a weight module  if the action of $d_0$ on $V$ is diagonalizable, i.e., $V=\bigoplus\limits_{\lambda\in \bC} V_{\lambda}$, where $V_{\lambda}=\{v\in V\,|\,d_0 v=\lambda v\}$.  The support set  of a weight module $V$ is defined by $\supp(V)=\{\lambda\in \bC\,|\,V_{\lambda}\ne 0\}$.   A weight $L$-module $V$  is called { cuspidal} or {uniformly bounded} if  there exists some $N\in \bN$ such that $\dim V_{\lambda}\le N,\forall \lambda\in \supp(V)$.

A weight $AW$-module $V$ is called a Jet $W$-module if $V$ is a free $A$-module of finite rank. It is known that any simple Jet $W$-module is isomorphic to an $AW$-module $V(\alpha,\beta,1)$ with a standard basis $\{e_i\,|\,i\in \bZ\}$ and the actions $d_i e_j=(\alpha+j+i\beta)e_{i+j}, t^i e_j=e_{i+j}$.

Let  $\wt{\fL}=W\ltimes (I\oplus A)$.  An $\wt{\fL}$-module $V$ is called an $A\fL$-module if $A$ acts associatively,  i.e., $t^0 v=v, fgv=f(gv),\forall f,g\in A, v\in V$.
A weight $A\fL$-module $V$ is  called a Jet $\fL$-module if $V$ is a free $A$-module of finite rank.  Denote by $\sJ$ the category of Jet $\fL$-modules. Since any module in $\sJ$ is a free $A$-module of finite rank, it is a finite direct sum of indecomposable modules and hence to determine $\sJ$ it suffices to classify indecomposable modules $V=\bigoplus\limits_{\lambda\in\bC}V_\lambda\in\sJ$. For any $v\in V_\lambda$,
\begin{align*}
&d_0d_jv=d_jd_0v+[d_0,d_j]v=(\lambda+j)d_jv,\\
&d_0x_s(j)v=x_s(j)d_0v+[d_0,x_s(j)]v=(\lambda+j)x_s(j)v,\\
&d_0t^jv=t^jd_0v+[d_0,t^j]v=(\lambda+j)t^jv.
\end{align*}
Thus if $V$ is indecomposable, then $\supp(V)=\lambda+\bZ$ for some $\lambda\in\bC$. For a fixed $\lambda\in\bC$, denote by $\sJ_\lambda$ the subcategory of $\sJ$ supported on $\lambda+\bZ$.

The following lemmas will be useful in our later discussion.

\begin{lemm}[{\cite[Lemma 2.2]{LX}}]\label{lx}
Let $A, B$ be unital associative superalgebras and $B$ have a countable basis.
If $V$ is a simple module over $A\otimes B$ that contains a strictly simple $B$-submodule $V_2$ , then $V\cong V_1\otimes V_2$
for a simple $A$-module $V_1$. Here, a $B$-module is called strictly simple if it is a simple module over the associative algebra $B$ forgetting the $\bZ_2$-grading.
\end{lemm}

\begin{lemm}[{\cite[Proposition 19.1]{H}}]\label{reductive}
\begin{enumerate}
\item Let $L$ be a finite dimensional reductive Lie algebra. Then $L=[L,L]\oplus Z(L)$ and $[L,L]$ is semisimple.
\item Let $L\subseteq\gl(V)$ ($\dim V<\infty$) be a Lie algebra acting irreducibly on $V$. Then $L$ is reductive and $\dim Z(L)\leq1$.
\end{enumerate}
\end{lemm}

\begin{lemm}[{\cite[Theorem 2.1]{M}}, Engel's Theorem for Lie superalgebras]\label{Engel}
Let $V$ be a finite dimensional module for the Lie superalgebra $L=L_{\bar{0}}\oplus L_{\bar{1}}$ such that the elements of $L_{\bar{0}}$ and $L_{\bar{1}}$ respectively are nilpotent endomorphisms of $V$. Then there exists a nonzero element $v\in V$ such that $xv=0$ for all $x\in L$.
\end{lemm}

We also need the following results in $\fL$.

\begin{lemm}\label{rel-subalg}
Let $k,\ell\in\bZ_+$ and $\{x_s\,|\,s\in S\}$ be the fixed basis of $\fg$. Then
\begin{align*}
&[(t-1)^kd_i,(t-1)^\ell d_j]=(l-k+j-i)(t-1)^{k+\ell}d_{i+j}+(l-k)(t-1)^{k+\ell-1}d_{i+j},\\
&[(t-1)^kd_i,(t-1)^\ell x_s(j)]=(j+i\beta_s)(t-1)^{k+\ell}x_s(i+j)+(\ell+k\beta_s)(t-1)^{k+\ell-1}x_s(i+j+1).
\end{align*}
\end{lemm}
\begin{proof}
Lemma follows from direct computations.
\end{proof}

From Lemma \ref{rel-subalg} and $[x(k),y(j)]=[x,y](k+j)$, we get
\begin{lemm}\label{ideal}
For $k\in\bZ_+$, let $\fa_k=(t-1)^{k+1}W\ltimes(\fg\otimes(t-1)^kA)$. Then
\begin{enumerate}
\item $\fa_0$ is a subalgebra of $\fL$;
\item $\fa_k$ is an ideal of $\fa_0$ and $\fa_0/\fa_1\cong\widehat{\fg}$;
\item $[\fa_1,\fa_k]\subseteq\fa_{k+1}$;
\item $[\fa_0,\fa_0]\supseteq\fa_1$.
\item The ideal generated by $(t-1)^kW$ contains $\fa_{k}$.
\end{enumerate}
\end{lemm}

Also, we have
\begin{lemm}\label{rao}
Any co-finite ideal of $(t-1)W$ contains $(t-1)^kW$ for large $k$.
\end{lemm}
\begin{proof}
It suffices to show any finite dimensional $(t-1)W$-module $V$ is annihilated by $(t-1)^kW$ for large $k$. Let $f(\lambda)$ be the characterization polynomial of $(t-1)d_{-1}$ on $V$. Then there exists some $k\in\bN$ such that $(f(\lambda-l),f(\lambda))=1$ for all $l>k$. From
\[f((t-1)d_{-1}-l)\cdot(t-1)^{l+1}d_{-1}\cdot v=(t-1)^{l+1}d_{-1}\cdot f((t-1)d_{-1})\cdot v=0, \forall v\in V,\]
we see that $(t-1)^{l+1}d_{-1}V=0$.  Hence, $2(t-1)^{l+3}d_{i-1}=[(t-1)^{l+2}d_{-1},(t-1)d_i]-[(t-1)^{l+1}d_{-1},(t-1)^2d_i]$ annihilates $V$ for all $i\in\bZ$ and $l>k$. This completes the proof.
\end{proof}

The following result is well-known.

\begin{lemm}\label{rough}
Let $M$ be a weight $W$-module with finite dimensional weight spaces and $\mathrm{supp}(M)\subseteq\lambda+\bZ$. If for any $v\in M$, there exists some $N(v)\in\bN$ such that $d_iv=0, \forall i\geq N(v)$, then $\mathrm{supp}(M)$ is upper bounded.
\end{lemm}

Let $\fL_R=\fL\otimes R$ be the map superalgebra with $R$ a Noetherian unital supercommutative associative superalgebra. The bracket in $\fL_R$ is
\[
[y_1\otimes r_1, y_2\otimes r_2]=(-1)^{|y_2||r_1|}[y_1,y_2]\otimes r_1r_2,
\]
where $y_1,y_2$ are homogeneous in $\fL$ and $r_1,r_2$ are homogeneous in $R$. The following theorem tells us that the classification of simple cuspidal modules is an important step for classification of simple quasifinite modules.
\begin{theo}
Suppose $|S|<\infty$ and $M$ is a simple $\fL_R$-weight module with finite dimensional weight spaces which is not cuspidal, then $M$ is a highest (lowest) weight module.
\end{theo}
\begin{proof}
Since $M$ is not cuspidal, there is a $k\in\bZ$ such that $\dim_{-k+\lambda}>2(1+|S|)(\dim M_\lambda+\dim M_{\lambda+1})$. Without lost of generality, we may assume that $k\in\bN$ and $R$ is generated by $r_1,\cdots,r_m$. Set $r_0=1$. Then there exists a homogeneous element $w\in M_{-k+\lambda}$ such that $(d_i\otimes r_j)w=(x_s(i)\otimes r_j)w=0, \forall i=k,k+1,j=0,1,\cdots,m, s\in S$. Hence,
\[
(d_i\otimes r)w=(x_s(i)\otimes r)w=0, \forall i\geq k^2, r\in R, s\in S.
\]

It is easy to check that $M'=\{v\in M\,|\,\mbox{there is some } n\in\bN\mbox{ such that } (d_i\otimes r)v=(x_s(i)\otimes r)v=0, \forall j>n, r\in R, s\in S\}$ is an $\fL_R$-module of $M$ containing $w$. Hence $M'=M$. From Lemma \ref{rough}, we know $\mathrm{supp}(M)$ is upper bounded, so $M$ is a highest weight module.
\end{proof}

\section{Jet modules for $\fL$}
\label{Jet}

In this section, we will study Jet modules for $\fL$.

Let $\cU=U(\wt{\fL})$ and ${\cI}$ be the left
ideal of $\cU $ generated by $t^i\cdot t^j-t^{i+j}$ and $t^0-1$ for all $i, j\in\mathbb{Z}$. Then it is clear that $\cI$ is  an ideal of $\cU$. Let $\ol{\cU}=\cU/\cI$. Then
the category of $A\fL$-modules is equivalent to the category of $\ol{\cU}$-modules.

Let $\cT$ be the subspace of $\ol{\cU}$ spanned by $t^{-i}\cdot d_i-d_0$ and $t^{-i}\cdot x_s(i)$ for all $i \in \mathbb{Z}$ and $s\in S$.

\begin{prop}
\begin{enumerate}
\item $[\cT,d_0]=[\cT,A]=0$.
\item $\cT$ is a Lie subsuperalgebra of $\ol{\cU}$.
\end{enumerate}
\end{prop}
\begin{proof}
The first statement follows from
\begin{align*}
&[d_0,t^{-i}\cdot d_i-d_0]=[d_0,t^{-i}]\cdot d_i+t^{-i}\cdot[d_0,d_i]=-it^{-i}\cdot d_i+it^{-i}\cdot d_i=0,\\
&[d_0,t^{-i}\cdot t^ix_s]=[d_0,t^{-i}]\cdot t^ix_s+t^{-i}\cdot[d_0,t^ix_s]=-it^{-i}\cdot t^ix_s+it^{-i}\cdot t^ix_s=0,\\
&[t^k,t^{-i}\cdot d_i-d_0]=t^{-i}\cdot[t^k,d_i]-[t^k,d_0]=-kt^k+kt^k=0,\\
&[t^k,t^{-i}\cdot t^ix_s]=0.
\end{align*}
And the second statement follows from
\begin{align*}
&[t^{-i}\cdot d_i-d_0, t^{-j}\cdot d_j-d_0]\\
=&[t^{-i}\cdot d_i, t^{-j}\cdot d_j]\\
=&[t^{-i}\cdot d_i, t^{-j}]\cdot d_j+t^{-j}\cdot[t^{-i}\cdot d_i, d_j]\\
=&t^{-i}\cdot[d_i, t^{-j}]\cdot d_j+[t^{-i}, t^{-j}]\cdot d_id_j+t^{-j}\cdot t^{-i}\cdot[d_i, d_j]+t^{-j}\cdot[t^{-i}, d_j]\cdot d_i\\
=&-jt^{-j}\cdot d_j+(j-i)t^{-(i+j)}\cdot d_{i+j}+it^{-i}\cdot d_i\\
=&-j(t^{-j}\cdot d_j-d_0)+(j-i)(t^{-(i+j)}\cdot d_{i+j}-d_0)+i(t^{-i}\cdot d_i-d_0),\\
\\
&[t^{-i}\cdot d_i-d_0,t^{-k}\cdot x_s(k)]\\
=&[t^{-i}\cdot d_i, t^{-k}\cdot x_s(k)]\\
=&[t^{-i}\cdot d_i, t^{-k}]\cdot x_s(k)+t^{-k}\cdot[t^{-i}\cdot d_i, x_s(k)]\\
=&t^{-i}\cdot[d_i, t^{-k}]\cdot x_s(k)+[t^{-i}, t^{-k}]\cdot d_ix_s(k)+t^{-k}\cdot t^{-i}\cdot [d_i, x_s(k)]+t^{-k}\cdot[t^{-i}, x_s(k)]\cdot d_i\\
=&-kt^{-k}\cdot x_s(k)+(k+i\beta_s)t^{-(k+i)}\cdot x_s(k+i),\\
\\
&[t^{-i}\cdot x_s(i),t^{-k}\cdot x_p(k)]=t^{-(i+k)}\cdot [x_s,x_p](k+i).\qedhere
\end{align*}
\end{proof}

Denote the simple associative algebra $\mathbb{C}[t^{\pm 1}, d_0]$ by $A[d_0]$.
\begin{prop}\label{isomU}
We have the associative superalgebra isomorphism $\ol{\cU}\cong A[d_0]\otimes U(\cT)$.
\end{prop}
\begin{proof}
Note that $U(\cT)$ is an associative subalgebra of $\ol{\cU}$. Define the map $\iota: A[d_0]\otimes U(\cT)\to\ol{\cU}$ by $\iota(t^id_0^j\otimes y)=t^i\cdot d_0^j\cdot y+\cI, \forall i\in\bZ, j\in\bZ_+, y\in U(\cT)$. Then the restriction of $\iota$ on $A[d_0]$ and $U(\cT)$ are well-defined homomorphisms of associative algebras, so $\iota$ is a well-defined homomorphism of associative algebras. From
\[\iota(t^i\otimes(t^{-i}\cdot d_i-d_0)+t^id_0\otimes1)=d_i,\,\, \iota(t^i\otimes(t^{-i}\cdot x_s(i)))=x_s(i),\]
we can see that $\iota$ is an epimorphism.

By PBW Theorem we know that $\ol{\cU}$ has a basis consisting monomials in variables $\{d_i, x_s(j)\,|\,i\in\bZ\setminus\{0\}, j\in\bZ,s\in S\}$ over $A[d_0]$. Therefore $\ol{\cU}$ has an $A[d_0]$-basis consisting monomials in the variables $\{t^{-i}\cdot d_i-d_0, t^{-j}\cdot x_s(j)\,|\, i\in\bZ\setminus\{0\}, j\in\bZ, s\in S\}$. So $\iota$ is injective and hence an isomorphism.
\end{proof}

\begin{prop}\label{isomKT}
For the Lie superalgebras $\cT$ and $\fa_0$, there is $\cT\cong \fa_0$.
\end{prop}
\begin{proof}
Obviously, $\fa_0$ is
spanned by $d_i-d_0$ and $x_s(i)$ for all $i\in \mathbb{Z},\ s\in S$. It is easy to check that the map $\cT\rightarrow \fa_0; t^{-i}\cdot d_i-d_0\mapsto d_i-d_0, t^{-i}\cdot x_s(i)\mapsto x_s(i)$ a Lie superalgebra isomorphism.
\end{proof}

The following proposition tells us that to study $\sJ_\lambda$, it suffices to study the category of finite dimensional modules for $\fa_0$.

\begin{prop}\label{equiva}
Let $\lambda\in\bC$. The category $\sJ_\lambda$ is equivalent to the category of finite dimensional modules for $\fa_0$.
\end{prop}
\begin{proof}
Let $V\in\sJ_\lambda$ and $\ol{V}=V_\lambda$. The invertible map $t^r: \ol{V}\to V_{\lambda+r}$ identifies all weight spaces with $\ol{V}$ and since $V$ is a free $A$-module it follows that any basis for $\ol{V}$ is also an $A$-basis for $V$. Furthermore the finite rank condition of $V$ implies that $\ol{V}$ is finite dimensional. Thus, $V\cong A\otimes \ol{V}$. From $[d_0,\cT]=0$ we know that $\cT\cdot \ol{V}\subseteq \ol{V}$, which means that $\ol{V}$ is a finite dimensional $\cT$-module. And hence $\ol{V}$ is a finite dimensional $\fa_0$-module by Proposition \ref{isomKT}.

Conversely, for any finite dimensional $\fa_0$-module $\ol{V}$, it is easy to check that $A\otimes \ol{V}\in\sJ_\lambda$ under the following actions: for $i,j\in\bZ,  u\in \ol{V}, s\in S$,
\begin{align}
&t^i\cdot(t^j\otimes u)=t^{i+j}\otimes u,\label{tensormodule1}\\
&d_i\cdot(t^j\otimes u)=(\lambda+j)t^{i+j}\otimes u+t^{i+j}\otimes(d_i-d_0)\cdot u,\label{tensormodule2}\\
&x_s(i)\cdot(t^j\otimes u)=t^{i+j}\otimes x_s(i)\cdot u.\label{tensormodule3}
\end{align}

So the category $\sJ_\lambda$ is equivalent to the category of finite dimensional $\fa_0$-modules.
\end{proof}

Denote  by $\cF_\lambda(\ol{V})$ the module $A\otimes\ol{V}$ defined by \eqref{tensormodule1}-\eqref{tensormodule3} for any $\fa_0$-module $\ol{V}$. Let $\ol{V}$ be a finite dimensional $\fa_0$-module, then it is also a finite dimensional module for $(t-1)W$, and therefore by Lemma \ref{rao}, there exists $k\in\bN$ such that $(t-1)^kW\cdot\ol{V}=0$. Hence, $\fa_{k}$ annihilates $\ol{V}$ by Lemma \ref{ideal}.

\begin{theo}\label{jet}
Let $\lambda\in\bC$.
\begin{enumerate}
\item Any $V\in\sJ_\lambda$ is isomorphic to some $\cF_\lambda(\ol{V})$, 
where $\ol{V}$ is a finite dimensional $\fa_0$-module annihilated by $\fa_k$ for some $k\in\bZ_+$.
\item The category of simple modules in $\sJ_\lambda$ is equivalent to the category of simple finite dimensional modules for $\widehat{\fg}$.
\end{enumerate}
\end{theo}
\begin{proof}
It remains to prove the second statement. Since $\fa_0/\fa_1\cong\widehat{\fg}$, by Proposition \ref{equiva}, from any simple finite dimensional $\widehat{\fg}$-module, one can get a simple module in $\sJ_\lambda$. Conversely, let $V$ be any simple module in $\sJ_\lambda$. From the first statement, we know that $V=A\otimes\ol{V}$ with $\ol{V}$ a simple finite dimensional $(\fa_0/\ann(\ol{V}))$-module with $\fa_k\subseteq\ann(\ol{V})$ for some $k\in\bZ_+$, here $\ann(\ol{V})$ is the ideal of $\fa_0$ that annihilates $\ol{V}$.

If $\fa_0/\ann(\ol{V})$ is a Lie algebra, then Lemma \ref{reductive} tells us that it is a reductive Lie algebra. And Lemma \ref{ideal} shows that $\fa_1+\ann(\ol{V})$ is a nilpotent ideal contained in $[\fa_0/\ann(\ol{V}),\fa_0/\ann(\ol{V})]$, which is semisimple. Hence, $\fa_1+\ann(\ol{V})$ acts trivially on $\ol{V}$. Thus, $\ol{V}$ is a simple finite dimensional module for $\fa_0/\fa_1\cong\widehat{\fg}$.

Now suppose $\fa_0/\ann(\ol{V})$ is a Lie superalgebra. Then $\ol{V}$ is a finite dimensional module for $(\fa_{0,\bar{0}}+\ann(\ol{V}))/\ann(\ol{V})$ and hence the even part of $\ol{\fa_1}=(\fa_1+\ann(\ol{V}))/\ann(\ol{V})$ acts nilpotently on $\ol{V}$. Since $[x,x]\in\ol{\fa_{1,\bar{0}}}, \forall x\in\ol{\fa_{1,\bar{1}}}$, every element in $\ol{\fa_{1,\bar{1}}}$ acts nilpotently on $\ol{V}$. Hence, by Lemma \ref{Engel}, there is nonzero $v\in\ol{V}$ annihilated by $\fa_1+\ann(\ol{V})$. And therefore $\fa_1\ol{V}=0$, which means $\ol{V}$ is a simple finite dimensional module for $\fa_0/\fa_1\cong\widehat{\fg}$.
\end{proof}

\section{Simple cuspidal modules for $\fL$}
\label{cuspL}

In this section, we will classify all simple cuspidal modules for $\fL$ using the method ``$A$-cover". Throughout this section, we assume that $\beta_s\neq1$ for all $s\in S$.

\begin{defi}
Let $M$ be an $\fL$-module and $\fri=\delta_{IM,0}W+(1-\delta_{IM,0})I$. An \emph{$A$-cover} of $M$ is the subspace $\widehat{M}\subset \mbox{Hom}(A, M)$
spanned by $\mu(x, u), \forall u\in M, x\in \fri$, where $\mu(x, u)\in \mathrm{Hom}(A, M)$ is given by
\begin{align*}
\mu(x, u)(a)&=(ax)u,\ \forall a\in A.
\end{align*}
\end{defi}

\begin{prop}\label{prop1}
\begin{enumerate}
\item $\widehat{M}$ is an $A\fL$-module under the following actions:  for all homogeneous $l\in \fL,\ x\in \fri$, and for all  $a\in A,\ u\in M,$
\begin{align}
l\mu(x, u)&=\mu([l, x], u)+(-1)^{|l||x|}\mu(x, lu),\label{eqc} \\
a\mu(x, u)&=\mu(ax, u).  \nonumber
\end{align}
\item If $M$ is a weight module, then so is $\widehat{M}$.
\item There exists a homomorphism $\pi: \widehat{M}\rightarrow M$ of $\fL$-module satisfying $\pi(\widehat{M})=\fri M$.
\end{enumerate}
\end{prop}

\begin{proof}
\begin{enumerate}
\item Direct computations show that $\widehat{M}$ is  both an $\fL$-module and an $A$-module. Hence $\widehat{M}$ is an $A\fL$-module follows from: for  homogeneous $ y\in I, z\in\fri$ and $x\in W, u\in M,a\in A$,
    \begin{multline*}
    x(a\mu(z,u))-a(x\mu(z,u))=x\mu(az,u)-a\mu([x,z],u)-a\mu(z,xu)\\
    =\mu([x,az],u)-\mu(a[x,z],u)=\mu([x,a]z,u)=[x,a]\mu(z,u);
    \end{multline*}
    \begin{multline*}
    y(a\mu(z,u))-a(y\mu(z,u))=y\mu(az,u)-a\mu([y,z],u)-(-1)^{|y||z|}a\mu(z,yu)\\
    =\mu([y,az],u)-\mu(a[y,z],u)=0.
    \end{multline*}
\item If $M$ is a weight module, then $\widehat{M}$ is spanned by $\mu(x, u)$ with $x$ and $u$
being eigenvectors with respect to the action of $d_0$. Then \eqref{eqc} shows that $\mu(x, u)$ is
also an eigenvector for the action of $d_0$. Since $\widehat{M}$ is spanned by its weight vectors, it is
a weight module.
\item Define $\pi: \widehat{M}\rightarrow
M$ by $\pi(\mu(x, u))=\mu(x, u)(1)=xu,\ \forall x\in \fri,\ u\in M.$
It is easy to see that
\begin{align*}
\pi(l\mu(x, u))& =  \pi(\mu([l, x], u)+(-1)^{|l||x|}\psi(x, lu))\\
&=  [l, x]u+(-1)^{|x||l|}x(lu) = l(xu)=l\pi(\mu(x, u)),
\end{align*}
for any homogeneous $l\in \fL$. The surjectivity of $\pi$ is obvious.\qedhere
\end{enumerate}
\end{proof}

Recall that in \cite{BF}, the authors show that every cuspidal $W$-module is annihilated by the operators $\Omega_{k,s}^{(m)}$ for $m$ large enough.

\begin{lemm}\label{Omegaoper}({\cite[Corollary 3.7]{BF}})
For every $\ell\in\bN$ there exists $m\in\bN$ such that for all $k, s\in\bZ$ the differentiators $\Omega_{k, s}^{(m)}=\sum\limits_{i=0}^m(-1)^i\binom{m}{i}d_{k-i}d_{s+i}$ annihilate every cuspidal $W$-module with a composition series of length $\ell$.
\end{lemm}

Let $M$ be a cuspidal $\fL$-module. Then $M$ is a cuspidal $W$-module and hence there exists $m\in\bN$ such that $\Omega_{k,p}^{(m)}M=0, \forall k,p\in\bZ$. Therefore, $[x_s(j),\Omega_{k,p}^{(m)}]M=0, \forall j,k,p\in\bZ, s\in S$. Thus, on $M$ we have
\begin{align}
0=&[x_s(j+1),\Omega_{k,p-1}^{(m)}]-2[x_s(j),\Omega_{k,p}^{(m)}]+[x_s(j-1),\Omega_{k,p+1}^{(m)}]\nonumber\\
&-[x_s(j),\Omega_{k+1,p-1}^{(m)}]+2[x_s(j-1),\Omega_{k+1,p}^{(m)}]-[x_s(j-2),\Omega_{k+1,p+1}^{(m)}]\nonumber\\
=&(1-\beta_s)\sum\limits_{i=0}^{m}(-1)^i\binom{m}{i}\Big(x_s(j+k+1-i)d_{p+i-1}-2x_s(k+j-i)d_{p+i}\nonumber\\
&+x_s(j+k-i-1)d_{p+i+1}\Big)\nonumber\\
=&(1-\beta_s)\sum\limits_{i=0}^{m+2}(-1)^i\binom{m+2}{i}x_s(j+k+1-i)d_{p-1+i}.\label{Omega}
\end{align}
Since $\beta_s\neq1$, we have
\begin{lemm}\label{omega}
Let $M$ be a cuspidal module for $\fL$. Then there exists $m\in\bN$ such that for all $j, p\in\bZ$ and $s\in S$, the operators
$\overline{\Omega}_{j,p,s}^{(m)}=\sum\limits_{i=0}^m(-1)^i\binom{m}{i}x_s(j-i)d_{p+i}$ annihilate $M$.
\end{lemm}

\begin{lemm}\label{cuspidalcover}
Suppose $\fg$ is finite dimensional. Let $M$ be a cuspidal module for the Lie algebra $\fL$. Then its $A$-cover $\widehat{M}$ is cuspidal.
\end{lemm}
\begin{proof}
The case when $IM=0$ is proved in \cite{BF}. Now suppose $IM\neq0$. Since $\widehat{M}$ is an $A$-module, it is sufficient to show that one of its weight spaces is
finite-dimensional. Fix a weight
$\alpha+p,\ p\in \mathbb{Z}$ and let us prove that
$\widehat{M}_{\alpha+p}=\mathrm{Span}\{\mu(x_s(p-k), M_{\alpha+k})\,|\,k\in\mathbb{Z}, s\in S\}$
is finite-dimensional. Assume that $\alpha=0$ when $\alpha+\mathbb{Z}=\mathbb{Z}$,
which means that $\alpha+p\neq 0, \forall p\in \mathbb{Z}$.

We claim that
\begin{equation}\label{1}
\widehat{M}_{\alpha+p}=\mathrm{Span}\{\mu(x_s(p-k), M_{\alpha+k})\,|\,|k|\leq \frac{m}{2}, s\in S\}.
\end{equation}
To prove this claim we need to show that for all $q\in \mathbb{Z}$ and for all $u\in M_{\alpha+q}$,
$\mu(x_s(p-q), u)$ belongs to the right hand side of \eqref{1}. We prove this by induction on $|q|$. Assume that this is true for $|k|<|q|$. Now let's prove that this is true for $q$. If $|q|\leq \frac{m}{2}$, the claim holds. If
$|q|>\frac{m}{2}$, we may assume $q<-\frac{m}{2}$. The proof for $q>\frac{m}{2}$ is similar. Since
$d_0$ acts on $M_{\alpha+q}$ with a nonzero scalar, we can write $u=d_0v$ for some $v\in M_{\alpha+q}$.
Then
\[\sum\limits_{i=0}^{m}(-1)^i\binom{m}{i}\mu(x_s(p-q-i), d_iv)(t^r)=\sum\limits_{i=0}^{m}(-1)^i\binom{m}{i}x_s(p+r-q-i)d_iv=\overline{\Omega}_{p+r-q, 0,s}^{(m)}v=0.\]
So we get $\sum\limits_{i=0}^{m}(-1)^i\binom{m}{i}\mu(x_s(p-q-i), d_iv)=0.$ Thus
\begin{equation}\label{2}
\mu(x_s(p-q), u)=\mu(x_s(p-q), d_0v)=-\sum\limits_{i=1}^{m}(-1)^i\binom{m}{i}\mu(x_s(p-q-i), d_iv).
\end{equation}
Since there is $|q+i|\leq \frac{m}{2}(1\leq i\leq m)$, the right hand side of \eqref{2} belongs to the right hand side of  \eqref{1} by induction assumption.
So does $\mu(x_s(p-q), u)$.

So lemma follows from the facts $\dim M_{\alpha+k}<\infty$ for any fixed $k$ and $|S|<\infty$.
\end{proof}

The following theorem gives a classification of simple cuspidal $\fL$-modules.
\begin{theo}\label{theorem1}
Any simple cuspidal $\fL$-module is a simple quotient of a simple cuspidal $A\fL$-module.
\end{theo}
\begin{proof}
The trivial module is a simple quotient of the simple cuspidal $A\fL$-module $A\otimes\bC$ with $\bC$ a trivial module for $\fL$. Now suppose $M$ is a non-trivial simple cuspidal $\fL$-module. If $IM=0$, then $M$ is a simple cuspidal $W$-module. So $M$ is a simple quotient of a simple cuspidal $AW$-module by Lemma 5.5 in \cite{BF}. Any $AW$-module is naturally an $A\fL$-module with trivial $I$-action since $I$ is an ideal. If $IM\neq 0$,
then $IM=M$ since $M$ is simple. Hence there is a surjective homomorphism
$\pi$ from the cuspidal $A\fL$-module $\widehat{M}$ to $M$ by Proposition \ref{prop1} and Lemma \ref{cuspidalcover}.

Consider a composition series of the $A\fL$-module $\widehat{M}$:
$$0=\widehat{M}_0\subset \widehat{M}_1\subset\cdots\subset \widehat{M}_{l-1}\subset \widehat{M}_l=\widehat{M}$$
with the quotients $\widehat{M}_i/\widehat{M}_{i-1}$ being simple $A\fL$-modules. Let $p$ be the smallest integer such
that $\pi(\widehat{M}_p)\neq 0$. Since $M$ is a simple $\fL$-module, we have $\pi(\widehat{M}_p)=M$
and $\pi(\widehat{M}_{p-1})=0$. So there is a surjective homomorphism of $\fL$-modules
$$\overline{\pi}: \widehat{M}_p/\widehat{M}_{p-1}\rightarrow M.$$
This completes the proof.
\end{proof}

\section{Simple cuspidal modules for map superalgebras}
\label{cuspmap}

In this section, we will classify simple cuspidal modules for the map superalgebra $\fL_R=\fL\otimes R$.

First, let us classify simple cuspidal $A_R\fL_R$-modules, where $A_R=A\otimes R$, i.e., modules for $\wt{\fL_R}=\wt{\fL}\otimes R$ with $A_R$ acting associatively. Write $t^ir$ for $t^i\otimes r\in A_R$ and $xr$ for $x\otimes r$ in $\fL_R$. Let $\ol{\cU}^R$ be the quotient algebra of $U(\wt{\fL_R})$ modulo the ideal generated by $t^{i}r_1\cdot t^jr_2-t^{i+j}r_1r_2$ and $t^0r-r$ for all $i,j\in\bZ$ and $r,r_1,r_2\in R$. Then the category of $A_R\fL_R$-modules is equivalent to the category of  $\ol{\cU}^R$-modules. Since $R$ is Noetherian, supercommutative and $[R,\ol{\cU}^R]=0$, we know that any element in $R$ acts as a scalar on any simple $\ol{\cU}^R$-module $V$, that is there is a homomorphism of algebras $\psi: R\to\bC$ such that $R$ acts on $V$ as $\psi(R)$. In particular, $\psi(R_{\bar{1}})=0, \psi(1)=1$.

Let $\ol{\cU}^{R,\psi}=\ol{\cU}^R/\langle r-\psi(r)\,|\, r\in R\rangle$, $\cT^{R,\psi}$ be the subspace of $\ol{\cU}^{R,\psi}$ spanned by $\{t^{-i}\cdot d_ir-\psi(r)d_0, t^{-i}\cdot x_s(i)r\,|\, r\in R, i\in\bZ, s\in S\}$.
\begin{prop}
\begin{enumerate}
\item $[\cT^{R,\psi},d_0]=[\cT^{R,\psi},A_R]=0$.
\item $\cT^{R,\psi}$ is a Lie subsuperalgebra of $\ol{\cU}^{R,\psi}$. Moreover, $\cT^{R,\psi}$ is isomorphic to the subalgebra $\cT_\psi=\fa_0+\fL\otimes\Ker\psi$ of $\fL_R$.
\end{enumerate}
\end{prop}
\begin{proof}
The first statement follows from direct computations.

The fact that $\cT^{R,\psi}$ is a Lie subsuperalgebra of $\ol{\cU}^{R,\psi}$ follows from
\begin{multline*}
[t^{-i}\cdot d_ir_1-\psi(r_1)d_0,t^{-j}\cdot d_jr_2-\psi(r_2)d_0]=i\psi(r_2)(t^{-i}\cdot d_ir_1-\psi(r_1)d_0)\\
-j\psi(r_1)(t^{-j}\cdot d_jr_2-\psi(r_2)d_0)+(j-i)(t^{-i-j}\cdot d_{i+j}r_1r_2-\psi(r_1r_2)d_0),
\end{multline*}
and
\begin{align*}
&[t^{-i}\cdot d_ir_1-\psi(r_1)d_0,t^{-j}\cdot x_s(j)r_2]=(j+i\beta_s)t^{-i-j}x_s(i+j)r_1r_2-j\psi(r_1)t^{-j}\cdot x_s(j)r_2,\\
&[t^{-i}\cdot x_s(i)r_1,t^{-j}\cdot x_p(j)r_2]=t^{-i-j}\cdot [x_s,x_p](i+j)r_1r_2.
\end{align*}
Moreover, $t^{-i}\cdot d_ir-\psi(r)d_0\mapsto d_ir-\psi(r)d_0,\, t^{-i}\cdot x_s(i)r\mapsto x_s(i)r$ gives an isomorphism of Lie superalgebras from $\cT^{R,\psi}$ to $\cT_\psi$.
\end{proof}

Similar to Proposition \ref{isomU}, we have
\begin{prop}\label{mapisomU}
We have the associative superalgebra isomorphism
$\ol{\cU}^{R,\psi}\cong A[d_0]\otimes U(\cT^{R,\psi})$.
\end{prop}

For any simple $\fL$-module $N$, one can define an $\fL_R$-module $N_\psi=N$ with $\psi: R\to\bC$ a homomorphism of algebras via $xr\cdot u=\psi(r)x\cdot u$.
Such a simple  $\fL_R$-module is called a (single point) evaluation module. When $N$ is cuspidal, we call $N_\psi$ a cuspidal evaluation module.

Let $V$ be any simple weight $A_R\fL_R$-module and $v\in V_\lambda$ be any nonzero weight vector. Then there exists $\psi: R\to\bC$ such that $R$ acts on $V$ as $\psi(R)$ and $Av_\lambda$ is a simple $A[d_0]$-submodule of $V$, hence $V\cong Av_\lambda\otimes\ol{V}$ with $\ol{V}$ a simple module for $U(\cT^{R,\psi})$ by Lemma \ref{lx}. Conversely, for any simple $\cT_\psi$-module $\ol{V}_\psi$, we can define an $A_R\fL_R$-module structure on $A\otimes\ol{V}_\psi$ as follows:
\begin{align*}
&t^kr\cdot t^i\otimes v=\psi(r)t^{i+k}\otimes v,\\
&d_kr\cdot t^i\otimes v=t^{i+k}\otimes(d_kr-\psi(r)d_0)v+\psi(r)(\lambda+i)t^{i+k}\otimes v,\\
&x_s(k)r\cdot t^i\otimes v=t^{k+j}\otimes x_s(k)r\otimes v.
\end{align*}
 Thus, we have
\begin{theo}\label{equivmap}
Any simple weight $A_R\fL_R$-module is isomorphic to $A\otimes \ol{V}_\psi$ for some simple $\cT_\psi$-module $\ol{V}_\psi$, where $\psi$ is a homomorphism from $R$ to $\bC$. In particular, any simple cuspidal $A_R\fL_R$-module is an evaluation $\fL_R$-module.
\end{theo}
\begin{proof}
It remains to show any simple cuspidal $A_R\fL_R$-module $V=A\otimes\ol{V}_\psi$ is an evaluation module, which is equivalent to show $\fL_{\Ker\psi}\ol{V}_\psi=0$.  $\ol{V}_\psi$ is finite dimensional since $V$ is cuspidal, and hence it is a finite dimensional module for $\fa_0$. So from Theorem \ref{jet}, $\ol{V}_\psi$ is annihilated by $\fa_k$ for some $k\in\bN$. Thus, theorem follows from the fact that the ideal generated by $\fa_k$ in $\cT_\psi$ is $\fL_{\Ker\psi}$.
\end{proof}

Now let us classify simple cuspidal $\fL_R$-module. Recall that in \cite{S}, the author showed that any simple cuspidal modules for $W_R$ with $R_{\bar{1}}=0$ is an evaluation module.
\begin{theo}[{\cite[Theorem 4.7]{S}}]\label{s}
For any simple cuspidal module $V$ over $W_R$ with $R_{\bar{1}}=0$, there exists some maximal ideal $\fJ\unlhd R$ such that $W_{\fJ}V=0$.
\end{theo}

From this theorem, we get the following property for cuspidal $W_R$-modules.
\begin{prop}\label{cuspidalwr}
For any cuspidal $W_R$-module $V$, there exists a co-finite ideal $\fJ\unlhd R$, such that $W_{\fJ}V=0$.
\end{prop}
\begin{proof}
Since $V$ is cuspidal, we have the composition series of $W_{R_{\bar{0}}}$-modules:
\[
0=V_0\subseteq V_1\subseteq\cdots\subseteq V_\ell=V,
\]
with $V_i/V_{i-1}$ simple cuspidal $W_{R_{\bar{0}}}$-modules. Hence, by Theorem \ref{s}, there are maximal ideals $\fJ_i\unlhd R_{\bar{0}}$ such that $W_{\fJ_i}V_i/V_{i-1}=0$. We claim that for $\fJ'=(\fJ_1\cdots \fJ_\ell)^\ell\unlhd R_{\bar{0}}$, $W_{\fJ'}V=0$. This follows from the fact $W$ is perfect and therefore any element in $W_{J'}$ can be written as $\sum[u_1{f_{11}\cdots f_{\ell1}},[u_2f_{12}\cdots f_{\ell2},[\cdots,[u_{i-1}f_{1,i-1}\cdots f_{\ell,i-1}, u_if_{1i}\cdots f_{\ell i}]]]], u_j\in W, f_{kj}\in \fJ_k$.

Clearly, $\fJ'$ is co-finite. Let $\fJ=\fJ'+\fJ'R_{\bar{1}}$. Then $\fJ$ is an ideal of $R$. Moreover, $[W_{\fJ'}, W_{\bar{1}}]=W_{\fJ'R_{\bar{1}}}$ tells us that $W_\fJ V=0$. To finish the proof, it remains to show $\fJ$ is co-finite. Since $R_{\bar{1}}$ is a finitely generated $R_{\bar{0}}$-module, we know that $R_{\bar{1}}/\fJ'R_{\bar{1}}$ is a finitely generated $R_{\bar{0}}/\fJ'$-module and hence is finite dimensional. Thus, $\fJ$ is co-finite.
\end{proof}

\begin{theo}\label{ev}
Any simple cuspidal $\fL_R$-module is a cuspidal evaluation module.
\end{theo}
\begin{proof}
Let $M$ be a simple cuspidal $\fL_R$-module. 
Then $M$ is a cuspidal $W_R$-module and hence $W_\fJ M=0$ for some co-finite ideal $\fJ\unlhd R$. Then $[W_\fJ, I]=I_\fJ$ shows that $\fL_\fJ M=0$. Hence $M$ is a cuspidal $\fL_{R/\fJ}$-module.

Since $M$ is a cuspidal $\fL$-module, we have the $A$-cover $\widehat{M}$, which is a cuspidal $A\fL$-module. Let $R'=R/\fJ$. Define the $A_{R'}\fL_{R'}$-module structure on $\widehat{M}\otimes R'$ as follows:
\begin{align*}
&xr\cdot (\mu(a,u)\otimes r')=(-1)^{(|a|+|u|)|r|}\mu([x,a],u)\otimes rr'+(-1)^{|a|(|x|+|r|)}\mu(a,xr\cdot u)\otimes r',\\
&fr\cdot(\mu(a,u)\otimes r')=(-1)^{|r|(|a|+|u|)}\mu(fa,u)\otimes rr'.
\end{align*}
We have a surjective $\fL_{R'}$-module homomorphism $\varphi:\widehat{M}\otimes R'\to M; \mu(a,u)\otimes r\mapsto (-1)^{|u||r|}ar\cdot u$. Hence, $M$ is a simple quotient of the $A_{R'}\fL_{R'}$-module $\widehat{M}\otimes R'$, which is cuspidal since $\widehat{M}$ is cuspidal and $R'$ is finite dimensional. Consider the composition series of the $A_{R'}\fL_{R'}$-module $\widehat{M}\otimes R'$:
\[
0=V_0\subseteq V_1\subseteq\cdots\subseteq V_\ell=\widehat{M}\otimes R',
\]
with $V_i/V_{i-1}$ being simple cuspidal $A_{R'}\fL_{R'}$-modules. Then there exists some $p\leq\ell$ such that $\varphi(V_p)=M$ and $\varphi(V_{p-1})=0$. That is, $M$ is a simple quotient of some simple cuspidal $A_{R'}\fL_{R'}$-module, and hence is an evaluation module by Theorem \ref{equivmap}.
\end{proof}
\section{Examples}
\label{example}

In this section, we will give some concrete examples.

\begin{exam}
The algebra $\mathrm{Vir}(0,\beta)$.
\end{exam}

The algebra $\mathrm{Vir}(0,\beta)$ is the central extension of $W\ltimes V(0,\beta,1)$ (cf. \cite{GJP}), it is an infinite dimensional Lie algebra with a basis $\{d_i,e_i, C_1,C_2,C_3,C_4\,|\,i\in\bZ\}$ satisfying
\begin{align*}
[d_i,d_j]&=(j-i)d_{i+j}+\delta_{i+j,0}\frac{i^3-i}{12}C_1,\\
[d_i,e_j]&=(j+i\beta)e_{i+j}+\delta_{i+j,0}(\delta_{\beta,0}(i^2+i)C_2+\delta_{\beta,-1}\frac{i^3-i}{12}C_2+\delta_{\beta,1}(iC_2+C_3)),\\
[e_i,e_j]&=i\delta_{i+j,0}\delta_{\beta,0}C_4.
\end{align*}
In particular, when $\beta\neq1$, the central extension is universal. And, $\mathrm{Vir}(0,-1)$ is the $W$-algebra $W(2,2)$ introduced in \cite{ZD} for the study of classification of vertex operator algebras generated by vectors of weight 2; while $\mathrm{Vir}(0,0)$ is the twisted Heisenberg-Virasoro algebra, the universal central extension of the Lie algebra of differential operators of order at most one (see \cite{A} for details). It is easy to check that $W\ltimes V(0,\beta,1)$ is isomorphic to the Lie algebra $\fL=W\ltimes(\fg\otimes A)$ with $\fg$ a $1$-dimensional Lie algebra.

Let $M$ be any simple cuspidal module for $\mathrm{Vir}(0,\beta)$. Then from the representation theory of Virsoro algebra (see \cite{KS}) and the representation theory of infinite dimensional Heisenberg Lie algebra (see \cite{LMZ}), we know that $C_1\cdot M=C_4\cdot M=0$. Thus, on $M$, for $p>>0$ we have
\begin{align*}
0&=\sum\limits_{i,j=0}^2(-1)^{i+j}\binom{2}{i}\binom{2}{j}[e_{m+2-i},\Omega_{-1+i,p-1+j}^{(m)}]\\
&=(-1)^m\Big(2\delta_{\beta,0}(2m+1)-\delta_{\beta,-1}\frac{3m+1}{4}-\delta_{\beta,1}\Big)C_2d_{p+m},
\end{align*}
which means $C_2\cdot M=0$. Thus, we have
\begin{prop}
Suppose $\beta\neq1$. The category of simple cuspidal $\mathrm{Vir}(0,\beta)$-modules is equivalent to the category of simple cuspidal $W\ltimes V(0,\beta,1)$-modules.
\end{prop}

Thus, by Theorem \ref{jet} and Theorem \ref{theorem1}, when $\beta\neq1$, any simple cuspidal $\mathrm{Vir}(0,\beta)$-module is a simple quotient of a simple cuspidal $A\fL$-module, which is of the form $\cF_\lambda(V)$ with $\lambda\in\bC$ and $V$ a simple finite dimensional $\widehat{\fg}$-module. Since $\fg=\bC d+\bC e$ with $[d,e]=\beta e$, we know that any simple finite dimensional $\widehat{\fg}$-module $V$ is $1$-dimensional. More precisely, we have $V=\bC u$ with $du=bu, eu=\delta_{\beta,0}Fu$. Therefore, any simple cuspidal $\mathrm{Vir}(0,\beta)$-module is a simple quotient of the module $\cF_\lambda(V)\in\sJ_\lambda$ with actions
\begin{align*}
d_i\cdot(t^j\otimes u)&=(\lambda+j+ib)t^{i+j}\otimes u,\\
e_i\cdot(t^j\otimes u)&=\delta_{\beta,0}F t^{i+j}\otimes u
\end{align*}
for $i,j\in\bZ$. The same results are given in \cite{LuZhao} when $\beta=0$ and in \cite{LiuZhu} when $\beta=-1$.

When $\beta=1$, $\bC[t^{\pm1}]$ is an $\fL$-module under
\begin{align*}
d_i\cdot t^j&=(\lambda+j+ib)t^{i+j},\\
e_i\cdot t^j&=\delta_{i,0}Ft^j,
\end{align*}
where $\lambda,b\in\bC, F\in\bC^*$, but it is not a quotient of $\cF_\lambda(V)$.

\begin{exam}
The Virasoro-Affine superalgebras.
\end{exam}

Let $\fg$ be a finite dimensional Lie algebra with a nondegenerate invariant normalized symmetric bilinear form $(\,,\,)$, then the affine-Virasoro Lie algebra (cf. \cite{Kac}) is the vector space $\widehat{\fL}=\fg\otimes\bC[t,t^{-1}]\oplus\bC C\oplus \bigoplus\limits_{m\in\bZ}\bC d_m$ with Lie brackets
\begin{align*}
[d_m,d_n]&=(n-m)d_{m+n}+\frac{m^3-m}{12}\delta_{m+n,0}C,\\
[d_m,x\otimes t^n]&=nx\otimes t^{m+i},\\
[x\otimes t^m,y\otimes t^n]&=[x,y]\otimes t^{m+n}+(x,y)m\delta_{m+n,0}C,
\end{align*}
where $x,y\in\fg, m,n\in\bZ$. Obviously, $\widehat{\fL}$ is the central extension of $W\ltimes(\fg\otimes A)$, where $\fg$ has a basis $\{x_1,\cdots,x_p\}$ and $\beta_i=0, \forall i$. Let $M$ be a simple cuspidal $\widehat{\fL}$-module. Then from the representation theory of Virasoro algebra, we have $CM=0$, that is the category of simple cuspidal $\widehat{\fL}$-modules is equivalent to the category of simple cuspidal $W\ltimes(\fg\otimes A)$-modules.  By Theorem \ref{jet} and Theorem \ref{theorem1}, we know that $M$ is a simple quotient of the module $\cF_\lambda(V)\in\sJ_\lambda$ for some simple finite dimensional $\fg$-module $V$ with actions
\begin{align*}
d_i\cdot (t^j\otimes u)&=(\lambda+j+ib)t^{i+j}\otimes u,\\
(x\otimes t^i)\cdot(t^j\otimes u)&=t^{i+j}\otimes(x\cdot u),
\end{align*}
where $x\in\fg, i,j\in\bZ, u\in V$. Moreover, $\cF_\lambda(V)$ is simple if and only if $V$ is nontrivial or $\lambda\notin\bZ$ or $b\neq0,1$.

The same results are given in \cite{GHL} when $\fg=\mathfrak{sl}_2$ or $\fg$ is the two dimensional solvable Lie algebra.

\begin{exam}
The differentially simple Lie superalgebra $W\otimes\Lambda(n)$.
\end{exam}

Let $W$ be the Witt algebra and $\Lambda(n)$ be the Grassmann superalgebra in $n$ variables $\xi_1,\xi_2,\cdots,\xi_n$. Then we have the differentially simple Lie superalgebra $W\otimes\Lambda(n)$ (see \cite{C}). Clearly $W\otimes\Lambda(n)$ is the map superalgebra $W_R$ with $R=\Lambda(n)$. By Theorem \ref{ev}, we know that any simple cuspidal $W\otimes\Lambda(n)$-module $M$ is an evaluation module. So $(d_k\otimes \xi_j)M=0$ for all $k\in\bZ, j=1,\cdots,n$. Hence, $M$ is a simple cuspidal $W$-module, that is, up to a parity-change, $M=\bC[t^{\pm1}]$ with $d_k\cdot t^i=(a+i+kb)t^{i+k}, \forall i,k\in\bZ$, for some $a,b\in\bC$.

\begin{exam}
The Lie superalgebra $\mathfrak{q}$.
\end{exam}

The Lie superalgebra $\fq$ is the Lie supersubalgebra of the $N=2$ Ramond algebra, which plays an important role for representation theory of the $N=2$ Ramond algebra (cf. \cite{LPX}). $\fq$ is a Lie superalgebra over $\bC$ with a basis $\{L_m,H_m,G_m, C\,|\, m\in\bZ\}$ and the following relations:
\begin{align*}
[L_m,L_n]&=(n-m)L_{m+n}+\frac{1}{12}\delta_{m+n,0}(m^3-m)C,\\
[L_m,H_n]&=nH_{m+n},\,\,[H_m,H_n]=\frac{1}{3}m\delta_{m+n,0}C,\\
[L_m,G_n]&=(n-\frac{m}{2})G_{m+n},\,\,[G_m,G_n]=0,\\
[C,\fq]&=0
\end{align*}
where $m,n\in\bZ$ and $L_m$'s, $H_m$'s are even elements and $G_m$'s are odd elements. Similarly, the category of simple cuspidal $\fq$-modules is equivalent to the category of simple cuspidal $\bar{\fq}$-modules, where $\bar{\fq}=\fq/\bC C$. Also, it is easy to see that $\bar{\fq}\cong W\ltimes(\fg\otimes A)$ with $\fg=\bC x+\bC y$ a two-dimensional Lie superalgebra with $x$ even, $y$ odd, and $[x,y]=y,[y,y]=0$. Let $\widehat{\fg}=\bC d+\fg$ with $[d,x]=0,[d,y]=-\frac{1}{2}y$. It is clear that any simple finite dimensional $\widehat{\fg}$-module is $1$-dimensional with $y$ acting trivially. Hence, by Theorem \ref{jet} and Theorem \ref{theorem1}, we know that the odd part of $\fq$ acts trivially on any simple cuspidal $\fq$-module, and hence any simple cuspidal $\fq$-module is a simple cuspidal module for its even part, which is the Heisenberg-Virasoro algebra. The same result were given in \cite{LPX}.

\begin{exam}
The Lie algebra $\mathfrak{wh}(\beta)=W\ltimes(H_3\otimes A)$.
\end{exam}

Let $\fg=H_3$ be the $3$-dimensional Heisenberg Lie algebra with basis $\{x,y,z\}$ and nontrivial brackets $[x,y]=z$. $\fg$ admits a derivation $d$ with $[d,x]=x, [d,y]=-y,[d,z]=0$, and $\widehat{\fg}$ is the Diamond Lie algebra, whose simple modules are classified in \cite{LPX1}. For any $\beta\in\bC$, consider the derivation $\beta d$, we have the Lie algebra $\mathfrak{wh}(\beta)=W\ltimes(H_3\otimes A)$ with
\begin{align*}
[d_i,x(j)]&=(j+i\beta)x(i+j),\\
[d_i,y(j)]&=(j-i\beta)x(i+j),\\
[d_i,z(j)]&=jz(i+j).
\end{align*}

From \cite{MP}, any simple finite dimensional $\widehat{\fg}$-module is $1$-dimensional with $H_3$ acting trivially, hence by Theorem \ref{jet} and Theorem \ref{theorem1}, we know that any simple cuspidal $\mathfrak{wh}(\beta)$-module with $\beta\neq\pm1$ is a simple cuspidal $W$-module.

\

Y. Cai: Department of Mathematics, Soochow University, Suzhou 215006, P. R. China. Email: yatsai@mail.ustc.edu.cn

R. L\"{u}: Department of Mathematics, Soochow University, Suzhou 215006, P. R. China. Email: rlu@suda.edu.cn

Y. Wang: Department of Mathematics Tianjin University Tianjin 300072, P. R. China. Email: wangyan09@tju.edu.cn


\end{document}